\documentclass[12pt]{article}
by-\headheight \advance \topmargin by-\headsep 
\usepackage{latexsym}
\usepackage{amsmath}
\usepackage{amssymb}
\usepackage{amsthm}
\usepackage{color}
\usepackage{colortbl}
\usepackage[square, comma, sort&compress, numbers]{natbib}
\newtheorem{thm}{Theorem}[section]

\newtheorem{lemma}[thm]{Lemma}

\newtheorem{corollary}[thm]{Corollary}

\theoremstyle{definition}
\newtheorem{definition}[thm]{Definition}

\numberwithin{equation}{section}

\newcommand{\R}{\mathbb{R}}

\newcommand{\N}{\mathbb{N}}

\newcommand{\biindice}[3]%
{

\begin{array}[t]{c}
#1\\
{\scriptstyle #2}\\
{\scriptstyle #3}\\
\end{array}

}

\def\biindice#1#2#3{#1_{\textstyle{#2\atop #3}}}

\newcommand{\ds}{\displaystyle}

\theoremstyle{plain}
\newtheorem{theorem}{Theorem}[section]

\numberwithin{equation}{section}

\begin{document}
\title{\bf Sign-changing solutions for the stationary Kirchhoff problems involving the fractional Laplacian in $\R^N$
\thanks{This work was supported  by the National Natural
	Science Foundation of China  (11371160, 11501231) and the Program for Changjiang Scholars and Innovative Research Team in University (\#IRT13066).
}}
\author{Kun Cheng \thanks {Department of
Mathematics, Central China Normal University, Wuhan, 430079, China
(chengkun0010@126.com).}
\and Qi Gao
\thanks {Department of Mathematics, School of Science, Wuhan University of Technology, Wuhan, 430070, China (gaoq@whut.edu.cn).}}

\date{\today}
\maketitle

\begin{abstract}
In this paper, we study the existence of least energy sign-changing solutions for a Kirchhoff-type problem involving the fractional Laplacian operator. By using the constraint variational method and quantitative deformation lemma, we obtain a least energy nodal solution $u_b$ for the given problem. Moreover, we show that the energy of $u_b$ is strictly larger than twice the ground state energy. We also give a convergence property of $u_b$ as $b\searrow 0$, where $b$ is regarded as a parameter.

\end{abstract}
{\bf Keywords:}  {Kirchhoff equation, Fractional Laplacian, Sign-changing solutions.}

\section{\label{Int}Introduction and main results}
Here we consider the existence of least energy sign-changing solutions for the Kirchhoff type problem involving a fractional Laplacian operator as below:
%More precisely, we study the following equations
\begin{equation}\label{eqs1.1}
\ds \left(a+b \iint_{\R^{2N}}\frac{|u(x)-u(y)|^2}{|x-y|^{n+2s}}dxdy\right) (-\Delta)^s u
+V(x)u=f(x,u)\ \  \mathrm{in}\   \R^{N},
\end{equation}
where $s\in(0,1)$ is fixed, $n>2s$, $a$ and $b$ are positive constants. $(-\Delta)^s$ is the fractional Laplacian operator which (up to normalization factors) may be defined along a function $\phi \in \mathcal{C}_0^{\infty}(\R^{N})$ as
\begin{equation*}
(-\Delta)^s \phi(x)=2\lim\limits_{\varepsilon \rightarrow 0^{+}}
\int_{\R^{N}\backslash B_{\varepsilon}(x)}
\frac{\phi(x)-\phi(y)}{|x-y|^{n+2s}}dy
\end{equation*}
for $x\in \R^{N}$, where $B_{\varepsilon}(x):=\{y\in \R^{N}:|x-y|<\varepsilon\}$.
For an elementary introduction to the fractional Laplacian and fractional Sobolev spaces we refer the reader
to \cite{DPV,MRS}. 

Throughout the paper, we make the following assumptions concerning the potential function $V(x)$:
\begin{enumerate}
	\item [($V_1$)]  \ $\displaystyle V \in \mathcal{C}(\R^N)$ satisfies $\inf_{x\in \R^N} V(x)\geq V_0>0$, where $V_0$ is a positive constant;
	\item [($V_2$)]  \  There exists $h>0$ such that $\displaystyle\lim_{|y|\rightarrow \infty}\text{meas}\left(\{x\in B_{h}(y):V(x)\leq c\}\right)=0$ for any $c>0$ ;
\end{enumerate}
where $B_{R}(x)$ denotes any open ball of $\R^N$ centered at $x$ with radius $R>0$, and $\text{meas}(A)$ denotes the Lebesgue measure of set $A$. 
The condition $(V_2)$, which is weaker than the coercivity assumption: $V(x)\rightarrow \infty$ as $|x|\rightarrow \infty$, was first introduced by Bartsch-Wang \cite{BW1} to overcome the lack of compactness.
Moreover, on the nonlinearity $f$ we assume
\begin{enumerate}
	\item [($f_1$)]  \ $\displaystyle f: \R^N \times \R\rightarrow \R$ is a $Carath\acute{e}odory$ function and $f(x,s)=o(|s|^3)$ as $s\rightarrow 0$;
	\item [($f_2$)]  \ For some constant $p \in (4, 2_s^*)$, $\displaystyle\lim_{s\rightarrow\infty}\frac{f(x,s)}{s^{p-1}}=0$, where $2_s^{*}=+\infty$ for $ N\le 2s$ and $2_s^*=\frac{2N}{N-2s}$ for $N>2s;$
	\item [($f_3$)]  \ $\lim\limits_{s\rightarrow \infty}\frac{F(x,s)}{s^{4}}=+\infty $ with $F(s)=\displaystyle\int_{0}^{s}f(t)dt$;
	\item [($f_4$)]  \ $\frac{f(s)}{|s|^3}$ is nondecreasing on $\R \setminus \{0\}$.
\end{enumerate}

The motivation to study problem \eqref{eqs1.1} comes from Kirchhoff equations of the type
\begin{equation}\label{eqs1.2}
\ds -\Big(a+b \int_{\Omega}|\nabla u|^2 dx\Big) \Delta u
=f(x,u)\ \  \mathrm{in}\   \Omega,
\end{equation}
where $\Omega\subset\mathbb{R}^N$ is a bounded domain, $a>0$, $b\ge0$ and $u$ satisfies some boundary conditions. The problem (\ref{eqs1.2}) is related to the stationary analogue of the Kirchhoff equation
\begin{equation}\label{eqs1.3}
\ds u_{tt}-\Big(a+b \int_{\Omega}|\nabla u|^2 dx\Big) \Delta u
=f(x,u)\ ,
\end{equation}
which was introduced by Kirchhoff \cite{K} as a generalization of the well-known D'Alembert wave equation
\begin{equation}\label{eqs1.4}
\ds \rho \frac{\partial ^2 u}{\partial t^2}
-\Big(\frac{p_0}{h}+
\frac{E}{2L} \int_{0}^{L}
\Big|\frac{\partial  u}{\partial x}\Big|^2 dx
\Big)=f(x,u)
\end{equation}
for free vibration of elastic strings. 
%Kirchhoff's model takes into account the changes in length of the string produced by transverse vibrations.
Here, $L$ is the length of the string, $h$ is the area of the cross section, $E$ is the Young modulus of the material, $\rho$ is mass density and $p_0$ is the initial tension. 
The problem \eqref{eqs1.4} has been proposed and studied as the fundamental equation for understanding several physical systems, where $u$
describes a process which depends on its average. It has also been used to model certain phenomena in biological systems
%In particular, the Kirchhoff equation (\ref{eqs1.4}) contains a nonlocal coefficient $p_0/h+(E/2L)\int_{0}^{L}|\frac{\partial  u}{\partial x}|^2dx$, which depends on the average $(1/L)\int_{0}^{L}|\frac{\partial  u}{\partial x}|^2dx$ of the kinetic energy $|\frac{\partial  u}{\partial x}|^2$ on $[0,L]$ and hence the equation is no longer a pointwise identity. We have to point out that such nonlocal problems also appear in other fields as biological systems, where $u$
%describes a process which depends on the average of itself, for example, population density 
(see \cite{CR}). 

After the pioneer work of Lions \cite{L}, where a functional analysis approach was proposed to Eq.\eqref{eqs1.3} with Dirichlet boundary condition, the Kirchhoff equations \eqref{eqs1.3} began to receive attention from many researchers. Recently, there are fruitful studies towards problem
(\ref{eqs1.2}), especially on the existence of global classical solution, positive solutions, multiple solutions and ground state solutions, see for example \cite{DS, ACM, HZ, LY} and the references therein.
%D'Ancona-Spagnolo \cite{DS} proved the existence of a global classical solution for a degenerate Kirchhoff equation with real analytic data. Kirchhoff's equation, in this case, is an example of a quasi-linear hyperbolic Cauchy problem that describes the transverse oscillations of a stretched string. Alves-Corr\^ea-Ma \cite{ACM} showed the existence of positive solutions to the class of nonlocal boundary value problems of quasilinear elliptic equation of Kirchhoff type via a variational framework. He-Zou \cite{HZ} showed the existence of infinitely many positive solutions to a class of Kirchhoff-type problems with Dirichlet boundary condition. Later on, via a monotonicity argument and a new version of the global compactness lemma, Li-Ye \cite{LY} proved the existence of positive ground state solution to a nonlinear problem of Kirchhoff type with power nonlinearities in $\mathbb{R}^3$.
%Especially, for the existence of positive solutions, multiple solutions, ground states and semiclassical states, we refer the interested readers to \cite{ACM,CDS,DS,HZ,LY}.
For sign-changing solutions, Zhang et al. \cite{ZP,MZ} used the method of invariant sets of descent flow to obtain the existence of sign-changing solution of (\ref{eqs1.2}).
By using the constraint variation methods and the quantitative deformation lemma, Shuai \cite{S}
studied the existence of least energy sign-changing solution for problem (\ref{eqs1.2}). For the other work about sign-changing solution of Kirchhoff type problem (\ref{eqs1.2}), we refer the reader to \cite{DPS,ML} and the reference therein.

On the other hand, we can also start the investigation to problem \eqref{eqs1.1} from the direction of nonlinear fractional Schr\"{o}dinger equation which was proposed by Laskin \cite{La1,La2}.
In recent years, there are plenty of research on the fractional Schr\"{o}dinger equation,

%When $b=0$, $a=1$, $V(x)=0$ and replace the integral domain $\R^N$ by a bounded smooth domain $\Omega \subset \R^N$, problem (\ref{eqs1.1}) can be
%reduced to the fractional Schr\"{o}dinger equation,
\begin{equation}\label{eqs1.5}
\left(-\Delta\right)^s u+V(x)u=f(x,u)\ \ \text{in}\ \ \mathbb{R}^N\ .
\end{equation}
%\begin{equation}\label{eqs1.5}
%\left \{ \begin{array}{ll}
%\ds (-\Delta)^s u
%=f(x,u)\ \  &\mathrm{in}\   \Omega,\\
%\ds u=0\ \  &\mathrm{in}\ \R^N\setminus\Omega \ ,
%\end{array} \right .
%\end{equation}
%where $\Omega$ is a bounded smooth domain in $\mathbb{R}^N$.
%where the nonlinearity $f$ satisfies some general conditions. This equation, which was proposed by Laskin \cite{La1,La2},
% is a fundamental equation in fractional quantum mechanics. 
%In recent years, a great attention has been focused on the study of nonlinear equations or systems involving fractional Laplacian operators. This type of operators arises in a quite natural way in many different applications, such as, continuum mechanics, phase transition phenomena, population dynamics, minimal surfaces, game theory and so on (see
%\cite{CS,CV,CM,La1,La2}). 
For the existence, multiplicity and behavior of solutions to Eq.\eqref{eqs1.5}, we refer to 
\cite{BCPS1,CS,CT,CS1,FW,RS,SV,T} and the references therein. In the remarkable work of Caffarelli-Silvestre \cite{CS}, the authors expressed the nonlocal operator $(-\Delta)^s$ as a Dirichlet-Neumann map for a certain elliptic
boundary value problem with local differential operators defined on the upper
half space. The technique in \cite{CS} is a powerful tool for the study of the equations
involving fractional operators.
%For the sign-changing solutions of (\ref{eqs1.5}),
Later on, by using this technique, Chang-Wang \cite{CW}  obtained a sign-changing solution via the invariant sets of descent flow.

Before presenting our main result, let us first recall the {\it usual} fractional Sobolev space
$H^{s}(\R^N)$ as follows
\begin{equation*}
H^{s}(\R^{N})=\Bigg\{
u\in L^2({\R^{N}}) \ \Big| \  \iint_{\R^{2N}} \frac{|u(x)-u(y)|^2}
{|x-y|^{N+2s}}dxdy<\infty
\Bigg\},
\end{equation*}
equipped with the inner product
\begin{equation*}
\langle u,v\rangle_{H^{s}(\R^{N})}=
\|u\|^{2}_{L^2({\R^{N}})} + \iint_{\R^{2N}} \frac{(u(x)-u(y))(v(x)-v(y))}
{|x-y|^{N+2s}}dxdy\ ,
\end{equation*}
and the corresponding norm
\begin{equation*}
\|u\|_{H^{s}(\R^{N})}=\Bigg(
\|u\|^{2}_{L^2({\R^{N}})} + \iint_{\R^{2N}} \frac{|u(x)-u(y)|^2}
{|x-y|^{N+2s}}dxdy
\Bigg)^{\frac12}.
\end{equation*}
It is well known that $(H^{s}(\R^{N}),\|u\|_{H^{s}(\R^{N})})$ is a uniformly convex Hilbert space and the embedding $H^{s}(\R^{N})  \hookrightarrow L^{q}(\R^{N})$ is continuous for any
$q \in [2,2_{s}^{*}]$ (See \cite{DPV}).
Moreover, in light of Proposition 3.2 in \cite{DPV}, for $u \in H^s(\R^N)$ we have
\begin{equation}\label{eqs1.6}
\|(-\Delta)^{s/2}u\|^2_{L^2{(\R^N)}}=
\iint_{\R^{2N}} \frac{|u(x)-u(y)|^2}
{|x-y|^{N+2s}}dxdy.
\end{equation}

In this paper, we denote the fractional Sobolev space for \eqref{eqs1.1} by
\begin{equation*}
H=\Big\{u\in H^{s}(\R^N)\Big|\ \int_{\R^N} V(x)u^{2}(x)dx<+\infty \Big\}\ ,
\end{equation*}
where the inner product is given by
\begin{equation}\label{eqs1.7}
\langle u,v\rangle_{H}=
\iint_{\R^{2N}} \frac{(u(x)-u(y))(v(x)-v(y))}
{|x-y|^{N+2s}}dxdy+
\frac{1}{a}\int_{\R^N} V(x)u(x)v(x)dx\ ,
\end{equation}
and the associated norm is
\begin{equation}\label{eqs1.8}
\|u\|_{H}=\Bigg(
\iint_{\R^{2N}} \frac{|u(x)-u(y)|^2}
{|x-y|^{N+2s}}dxdy
+\frac{1}{a}\int_{\R^{N}}V(x)|u(x)|^2 dx
\Bigg)^{\frac12}.
\end{equation}
It is readily seen that $(H,\| \cdot \|_{H})$ is
a uniformly convex Hilbert space and $\mathcal{C}_0^{\infty}(\R^N) \subset H$ (see \cite{PXZ}). Moreover, thanks to (\ref{eqs1.6}), the norm $\|\cdot\|_H$ can be also written as
\begin{equation}\label{eqs1.9}
\|u\|_{H}=\Bigg(
\int_{\R^{N}} |(-\Delta)^{s/2}u|^2 dx
+\frac{1}{a}\int_{\R^{N}}V(x)u(x)^2 dx
\Bigg)^{\frac12}.
\end{equation}

%And for potential function $V$ we have the following conditions:
%\begin{enumerate}
%	\item [($V_1$)]  \ $\displaystyle V \in \mathcal{C}(\R^N)$ satisfies $\inf_{x\in \R^N} V(x)\geq V_0>0$, where $V_0>0$ is a constant;
%	\item [($V_2$)]  \  There exists $h>0$ such that $\displaystyle\lim_{|y|\rightarrow \infty}\text{meas}(\{x\in B_{h}(y):V(x)\leq c\})=0$ for any $c>0$ ;
%\end{enumerate}
%where $B_{R}(x)$ denotes any open ball of $\R^N$ centered at $x$ with radius $R>0$, and $\text{meas}(A)$ denotes the Lebesgue measure of set $A$. The condition $(V_2)$, which is weaker than coercivity assumption: $V(x)\rightarrow \infty$ as $|x|\rightarrow \infty$, was firstly introduced by Bartsch-Wang \cite{BW1} to overcome the lack of compactness.
%
%Moreover, on the nonlinearity $f$ we assume
%\begin{enumerate}
%	\item [($f_1$)]  \ $\displaystyle f: \R^N \times \R\rightarrow \R$ is a $Carath\acute{e}odory$ function and $f(x,s)=o(|s|^3)$ as $s\rightarrow 0$;
%	\item [($f_2$)]  \ For some constant $p \in (4, 2_s^*)$, $\lim_{s\rightarrow\infty}\frac{f(x,s)}{s^{p-1}}=0$, where $2_s^{*}=+\infty$ for $ N\le 2s$ and $2_s^*=\frac{2N}{N-2s}$ for $N>2s;$
%	\item [($f_3$)]  \ $\lim\limits_{s\rightarrow \infty}\frac{F(x,s)}{s^{4}}=+\infty $ with $F(s)=\int_{0}^{s}f(t)dt$;
%	\item [($f_4$)]  \ $\frac{f(s)}{|s|^3}$ is nondecreasing on $\R \setminus \{0\}$.
%\end{enumerate}

Now we give the definition of a {\it weak} solution to \eqref{eqs1.1}:
\begin{definition}\label{D1.1}
    We say that $u\in H$ is a {\it weak} solution of \eqref{eqs1.1}, if
	\begin{align}
	\Big(a+
	b\int_{\R^{N}} |(-\Delta)^{s/2}u|^2 dx\Big)\int_{\R^{N}}&
	(-\Delta)^{s/2}u (-\Delta)^{s/2}\varphi dx\notag\\
	&+\int_{\R^N}V(x)u(x)\varphi(x)dx
	=\int_{\R^N}f(x,u)\varphi(x)dx\ ,
	\end{align}
for all $\varphi \in H$.
\end{definition}
And we will omit {\it weak} throughout this paper for convenience. Define the corresponding functional energy $I_b:H \rightarrow \R$ to problem (\ref{eqs1.1}) as below:
\begin{multline}\label{eqs1.11}
I_b(u)=\frac{a}{2}\int_{\R^{N}}
|(-\Delta)^{s/2}u|^2 dx
+\frac{b}{4}\Big(\int_{\R^{N}}
|(-\Delta)^{s/2}u|^2 dx
\Big)^2 \\
+\frac12 \int_{\R^N}V(x)u^2dx
-\int_{\R^N}F(x,u)dx.
\end{multline}
It is easy to see that $I_b$ belongs to  $\mathcal{C}^1 (H,\R)$ and the critical points of $I_b$
are the solutions of (\ref{eqs1.1}). Furthermore, if $u\in H$ is a solution of (\ref{eqs1.1}) and $u^{\pm}\neq 0$, then $u$ is a sign-changing solution of (\ref{eqs1.1}), where
$$
u^+(x)=\max \{u(x),0\}\quad \mathrm{and} \quad
u^-(x)=\min \{u(x),0\}.
$$

Our goal in this paper is to seek the least energy sign-changing solutions of (\ref{eqs1.1}). When $s=1$, $b=0$ and $a=1$, Eq. \eqref{eqs1.1} turns out to be \eqref{eqs1.5} mentioned above. There are several ways in the literature to obtain sign-changing solution for \eqref{eqs1.5}
(see \cite{BLW,BW1,BW2,NW}). %Bartsch-Willem \cite{BW2} obtained infinitely many node solutions by constructive arguments.
With the application of the minimax arguments in the presence of invariant sets of a descending flow, Bartsch-Liu-Weth \cite{BLW} obtained a sign-changing solution for (\ref{eqs1.5}) when $f$ satisfies the classical Ambrosetti-Rabinowitz condition \cite{AR} and conditions imposed on $V(x)$ to ensure the compact embedding. %Nodal solutions of (\ref{eqs1.12}) in bounded domains are proved to exist by Noussair and Wei \cite{NW} via the Ekeland variational principle and the implicit function theorem, and by Bartsch and Weth in \cite{BWe} based on the variational method together with the Brouwer degree theory. 
However, all these methods heavily rely on the following two decompositions:
\begin{equation}\label{eqs1.13}
J(u)=J(u^+)+J(u^-),
\end{equation}
\begin{equation}\label{eqs1.14}
\langle J'(u),u^+ \rangle=\langle J'(u^+),u^+ \rangle
\quad \mathrm{and}\quad
\langle J'(u),u^- \rangle=\langle J'(u^-),u^- \rangle\ ,
\end{equation}
where $J$ is the energy functional of (\ref{eqs1.5})
given by
\begin{equation}\label{eqs1.15}
J(u)=\frac12 \int_{\R^N}\left( |\nabla u|^2 +V(x)u^2\right)dx
-\int_{\R^N}F(x,u)dx.
\end{equation}
When $b>0$, due to the nonlocal term
$(a+b\int_{\R^N}|(-\Delta)^{s/2}u|^2dx)(-\Delta)^{s}u$, the functional do not possesses the same decompositions as (\ref{eqs1.13}) and (\ref{eqs1.14}). Indeed, since $\langle u^+,u^- \rangle _{H^{s}(\R^N)}>0$ when $u^{\pm}\neq 0$, a straightforward computation yields that
\begin{equation}\label{eqs1.16}
I_b(u)>I_b(u^+)+I_b(u^-),
\end{equation}
\begin{equation}\label{eqs1.17}
\langle I'_b(u),u^+ \rangle>\langle I'_b(u^+),u^+ \rangle
\quad \mathrm{and}\quad
\langle I'_b(u),u^- \rangle>\langle I'_b(u^-),u^- \rangle.
\end{equation}
Therefore, the methods to obtain sign-changing solutions for the local problem (\ref{eqs1.5}) seem not be applicable to problem (\ref{eqs1.1}). Motivated by the work in \cite{BWe}, to get
least energy sign-changing solutions, we define the following constrained set:
\begin{equation}\label{eqs1.18}
\mathcal{N}_{nod}^b=\{u \in H \ | \ u^\pm \neq 0,\langle I'_b(u),u^ \pm\rangle =0
\}, 
\end{equation}
and consider a minimization problem of $I_b$ on $\mathcal{N}_{nod}^b$. Comparing with the work in \cite{S}, the nonlocal term in (\ref{eqs1.1}), which can be regarded as a combination of the fractional Laplacian operator $(-\Delta)^{s}$ and the nonlocal term $(\int_{\R^N}|\nabla u|^2dx)\Delta u$ that appeared in general Kirchhoff type problem (\ref{eqs1.2}),
becomes even more complicated. This will cause some difficulties in proving that $\mathcal{N}_{nod}^b$ is nonempty. Indeed, Shuai \cite{S} proved $\mathcal{N}_{nod}^b\neq \varnothing$ by using the parametric method
and the implicit theorem. However, it seems that the method in \cite{S} is not suitable for problem (\ref{eqs1.1}) due to the complexity of the nonlocal term there.
Fortunately, inspired by \cite{AN}, we prove $\mathcal{N}_{nod}^b\neq \varnothing$ via a modified Miranda's theorem (see Lemma~\ref{lm2.3}). Eventually, we prove that the minimizer of the constrained problem is also a sign-changing solution via the quantitative deformation lemma (see \cite{WM}) and degree theory (see \cite{BWe}). 

The main results can be stated as follows.
\begin{theorem}\label{th1.1}
Let $(f_1)$--$(f_4)$ hold. Then problem (1.1) possesses one least energy sign-changing solution $u_b$.
\end{theorem}

Another goal of this paper is to prove that the energy of any sign-changing solution of \eqref{eqs1.1} is strictly larger than twice the ground state energy, this property is so-called energy doubling by Weth \cite{W}. Consider the semilinear equation (\ref{eqs1.5}), the conclusion is trivial. Indeed, we denote the Nehari manifold associated to (\ref{eqs1.5}) by
\begin{equation*}
\mathcal{N}=\{u \in H^1(\R^N) \setminus\{0\} |\ \langle J'(u),u\rangle=0\}
\end{equation*}
and define
\begin{equation}\label{eqs1.19}
c:=\inf_{u \in \mathcal{N}}J(u).
\end{equation}
Then, for any sign-changing solution $u\in H$ for equation (\ref{eqs1.5}), it is easy to show that $u^{\pm}\in \mathcal{N}$. Moreover, if the nonlinearity  $f(x,s)$ satisfies some conditions (see \cite{BWe}) which is analogous to $(f_1)$-$(f_4)$, we can deduce that
\begin{equation}\label{eqs1.20}
J(w)=J(w^+)+J(w^-)\geq 2 c.
\end{equation}
We point out that the minimizer of (\ref{eqs1.19}) is indeed a ground state solution of the problem
(\ref{eqs1.5}) and $c>0$ is  the least energy of all weak solutions of (\ref{eqs1.5}).
Therefore, by \eqref{eqs1.20}, it implies that the energy of any sign-changing solution of (\ref{eqs1.5}) is larger than
twice the least energy. When $s=1$ and $b>0$, a similar result was obtained by Shuai \cite{S} in a bounded domain $\Omega$. If $0<s<1$ and $\Omega= \R^N$, we are interested in that whether the property \eqref{eqs1.20} is still true for $I_b$. To answer this question, we denote the following Nehari manifold associated to (\ref{eqs1.1}) by
\begin{equation*}
\mathcal{N}^b=\{u \in H\setminus\{0\} |\ \langle I'_b(u),u\rangle=0\},
\end{equation*}
and define
\begin{equation*}
c^b:=\inf_{u \in \mathcal{N}^b} I_b(u).
\end{equation*}
Then we give the answer via the following theorem:

\begin{theorem}\label{th1.2}
	Assume $(f_1)-(f_4)$ hold, then $c^b>0$ is achieved and
	\begin{equation}\label{eqs1.21}
	I_{b}(u_b)>2c^b\ ,
	\end{equation}
	where $u_b$ is the least energy sign-changing solution obtained in Theorem 1.1. In particular, $c^b$ is
	achieved either by a positive or a negative function.
\end{theorem}
It is obvious that the energy of the sign-changing solution $u_b$ obtained in Theorem 1.1 depends
on $b$. It will be of interest to make this dependency precise. This will be the theme of our next result. Namely,
we give a convergence property of $u_b$ as $b\rightarrow 0$, which indicates
some relationship between $b>0$ and $b = 0$ for problem \eqref{eqs1.1}.

\begin{theorem}\label{th1.3}
	Let $f$ satisfies $(f_1)-(f_4)$. For any sequence $\{b_n\}$ with $b_n\rightarrow 0$ as $n\rightarrow \infty$, there exists a subsequence, still denoted by $\{b_n\}$, such that $u_{b_n}$ converges to $u_0$ strongly in $H$ as $n\rightarrow \infty$, where $u_0$ is a least energy sign-changing solution to the following problem
	\begin{equation}\label{eqs1.22}
	-a (\Delta)^s u+V(x)u=f(x,u)\ \  \mathrm{in}\   \R^{N}.
	\end{equation}
\end{theorem}

%We may point out that, to the best of our knowledge, there are few papers in the literature which considered (\ref{eqs1.1}). In \cite{FV}, Fiscell and Valdinoci  proposed  a stationary Kirchhoff variational model, which describes the nonlocal aspect of the tension arising from nonlocal measurements
%of the fractional length of string.
%Pucci et al. \cite{PXZ} considered a nonhomogeneous Schr\"{o}dinger-Kirchhoff type problem in $\R^N$ involving the fractional $p$-Laplacian. By using Ekeland variational principle and the Mountain Pass theorem, they established the multiplicity of solutions.
%In \cite{XB}, Xiang et al. studied a class of Kirchhoff nonlocal fractional equation in a bounded domain $\Omega$, and obtained infinitely many solutions by using symmetric mountain pass theorem and the Krasnoselskii genus theory. We also refer \cite{AFP,N}
%for some recent results in this direction.

\vskip 0.2cm

The plan of this paper is as follows: Section 2 reviews some prelimimary lemmas, Section 3 covers the proof of the achievement of least energy for the contraint problem \eqref{eqs3.20}, and Section 4 is devoted to the proofs of our main theorems.

\section{Preliminaries}

In this section, we first present a modified embedding lemma which is similar to Lemma 2.1 and Theorem 2.1 in \cite{PXZ}, and we omit the proof here.
\begin{lemma}\label{lm2.1}
(i) Suppose that $(V_1)$ holds. Let $q \in [2,2_{s}^{*}]$, then the embeddings
$$
H \hookrightarrow H^{s}(\R^{N}) \hookrightarrow L^{q}(\R^{N})
$$
are continuous, with $min\{1, V_0\}\|u\|^{2}_{{H^{s}}(\R^N)} \le \|u\|^{2}_{H}$ for all $u \in H$. In particular, there exists a constant $C_{q}>0$ such that
\begin{equation}\label{eqS2.1}
\|u\|_{L^{q}(\R^{N})} \le C_{q}\|u\|_{H}\ \ for\ all\  u\in H.
\end{equation}
Moreover, if $q \in [1,2^{*}_s)$, then the embedding $H \hookrightarrow \hookrightarrow L^{q}(B_R)$ is compact for any $R>0$.

(ii) Suppose that $(V_1)-(V_2)$ hold. Let $q \in [2,2_{s}^{*})$ be fixed and $\{u_n\}$ be a bounded sequence in H, then there exists $u \in H \cap L^{q}(\R^{N})$ such that, up to a subsequence,
$$
u_n \rightarrow u\ \ strongly\ in\ L^{q}(\R^{N})\ as\ n\rightarrow \infty.
$$
\end{lemma}

Now we recall the Miranda Theorem (see \cite{M1}):
\begin{lemma}\label{lm2.2}
Let $G=\{x=(x_1, x_2, \dots, x_N ) \in \R^{N}:\ |x_i|<L,\ for\ 1\le i \le N \}$. Suppose that the mapping $F=(f_1,f_2,\cdots,f_n):\overline{G}\rightarrow \R^{N}$ is continuous on the closure $\overline{G}$ such that $F(x) \neq \theta :=(0,0,\cdots,0)$ for $x$ on the boundary $\partial{G}$, and
$$
f_i (x_1,x_2,\cdots,x_{i-1},-L,x_{i+1},\cdots,x_N)
\ge 0 \ for\ 1\le i \le N,
$$
$$
f_i (x_1,x_2,\cdots,x_{i-1},+L,x_{i+1},\cdots,x_N)
\le 0 \ for\ 1\le i \le N,
$$
then $F(x)=0$ has a solution in $G$.
\end{lemma}

Inspired by the Miranda Theorem above, we have the following variant of Miranda Lemma:

\begin{lemma}\label{lm2.3}
	
Let $\mathcal{G}=\{x=(x_1, x_2, \dots, x_N ) \in \R^{N}:\ r<x_i<R,\ for\ 1\le i \le N \}$. Suppose that the mapping $F=(f_1,f_2,\cdots,f_n):\overline{\mathcal{G}} \rightarrow \R^{N}$ is continuous on the closure $\overline{\mathcal{G}}$ such that $F(x) \neq \theta:=(0,0,\cdots,0)$ for $x$ on the boundary $\partial{\mathcal{G}}$, and
$$
f_i (x_1,x_2,\cdots,x_{i-1},r,x_{i+1},\cdots,x_N)
\ge 0 \ for\ 1\le i \le N,
$$
$$
f_i (x_1,x_2,\cdots,x_{i-1},R,x_{i+1},\cdots,x_N)
\le 0 \ for\ 1\le i \le N.
$$
Then $F(x)=0$ has a solution in $\mathcal{G}$.
	
\end{lemma}
\begin{proof}
Consider the homotopy
$$
H: \overline{\mathcal{G}}\times[0,1]\subseteq \R^{N+1} \rightarrow \R^{N}
$$
by $H(x,t)=(1-t)F(x)+t(-x+\vec{a})$ for $\vec{a}=(a_1, a_2, \dots,a_N )$ with $r<a_i<R$, $1\le i\le N$. Then $H(x,t)\neq \theta$ for $x\in \partial{\mathcal{G}}$ and $t\in [0,1].$ In fact, $H(x,0)=F(x)\neq \theta$ since $\theta \notin F(\partial{\mathcal{G}})$, while $H(x,1)=-x+\vec{a} \neq \theta$ since $r<a_i<R$, $1\le i\le N$. Finally, we have $H(x,t)=\theta$ for some $t\in (0,1)$ which in turn implies that $F(x)+t(1-t)^{-1}(-x+a)=\theta$. But this contradicts the conditions of $f_i$, $i=1,\cdots,N$, given by the lemma. Therefore, by the homotopy invariant theorem of the degree theory, it follows that
$$
\mathrm{deg}[F,\mathcal{G},\theta]=
\mathrm{deg}[H(\cdot,0),\mathcal{G},\theta]=
\mathrm{deg}[H(\cdot,1),\mathcal{G},\theta],
$$
where $\mathrm{deg}[F,\mathcal{G},\theta]$ denotes the topological degree of $F$ at $\theta$ related to $\mathcal{G}$.
Hence $|\mathrm{deg}[F,\mathcal{G},\theta]|=1 \neq 0$ and the result follows by the Kronecker existence theorem.
\end{proof}

\section{\label{sec3} Minimizer of constraint minimization problem}
We consider a constraint minimization problem on $\mathcal{N}^b_{nod}$ (defined in \eqref{eqs1.18})
to seek a critical point of $I_b$ in this section. We begin this section by showing that the set $\mathcal{N}^b_{nod}$ is nonempty. Since the nonlocal term $\int_{\R^{N}}|(-\Delta)^{s/2}u|^2 dx (-\Delta)^{s} u$ is much more complicated than $\int_{\R^{N}}|\nabla u|^2 dx \Delta u$, the method in \cite{S} is no longer applicable here. Therefore, we take a different route, namely, we make use of Miranda theorem.

As a start, we define
$$
A^+ (u):
=\int_{\R^{N}}|(-\Delta)^{s/2}(u^{+})|^2dx,
\quad
A^- (u):
=\int_{\R^{N}}|(-\Delta)^{s/2}(u^{-})|^2dx,
$$
$$
B^+ (u):
=\int_{\R^{N}}V(x)(u^{+})^2 dx,
\quad
B^- (u):
=\int_{\R^{N}}V(x)(u^{-})^2 dx,
$$
$$
C(u):
=\int_{\R^{N}}
(-\Delta)^{s/2}(u^{+})(-\Delta)^{s/2}(u^{-})
dx.
$$

\begin{lemma}\label{lm3.1}
	Assume that $(f_1)-(f_4)$ hold. Let $u \in H$ with $u^{\pm}\neq 0$, then there is a unique pair $(\alpha_u,\beta_u)$ of positive numbers such that
	$\alpha_u u^+ + \beta_u u^- \in \mathcal{N}^b_{nod}$.
\end{lemma}
\begin{proof}
Fix an $u\in H$ with $u^{\pm}\neq 0$.
We first establish the existence of $\alpha_u$ and $\beta_u$. Consider the vector field
\begin{equation*}
W(\alpha,\beta)=
\big(\langle I'_b(\alpha u^+ + \beta u^-),\alpha u^+ \rangle,
\langle I'_b(\alpha u^+ + \beta u^-),\beta u^-\rangle
\big)
\end{equation*}
for $\alpha,\beta>0$, where
\begin{align*}
\langle
I'_b(\alpha &u^+ + \beta u^-),\alpha u^+ \rangle
\\
&=\ds a \int_{\R^{N}}(-\Delta)^{s/2}(\alpha u^+ + \beta u^-) (-\Delta)^{s/2}(\alpha u^+)dx+\int_{\R^{N}}V(x)(\alpha u^{+})^2 dx\\
&\quad+b \int_{\R^{N}}|(-\Delta)^{s/2}(\alpha u^+ + \beta u^-)|^2dx\int_{\R^{N}}(-\Delta)^{s/2}(\alpha u^+ + \beta u^-) (-\Delta)^{s/2}(\alpha u^+)dx
\\
&\quad -\int_{\R^{N}}f(x,\alpha u^+)\alpha u^+ dx\ ,
\end{align*}
and
\begin{align*}
\langle
I'_b(\alpha &u^+ + \beta u^-),\beta u^- \rangle
\\
&\ds= a \int_{\R^{N}}(-\Delta)^{s/2}(\alpha u^+ + \beta u^-) (-\Delta)^{s/2}(\beta u^-)dx+\int_{\R^{N}}V(x)(\beta u^-)^2 dx\\
&\quad+b \int_{\R^{N}}|(-\Delta)^{s/2}(\alpha u^+ + \beta u^-)|^2dx\int_{\R^{N}}(-\Delta)^{s/2}(\alpha u^+ + \beta u^-) (-\Delta)^{s/2}(\beta u^-)dx \\
&\quad-\int_{\R^{N}}f(x,\beta u^-)\beta u^- dx.
\end{align*}
By a straightforward computation, we get
\begin{align}
\langle I'_b(\alpha u^+ + \beta u^-),\alpha u^+ \rangle
&=b(\alpha^2 A^+ (u)+2\alpha\beta C(u)+\beta^2 A^- (u))(\alpha^2 A^+ (u)+\alpha\beta C(u))\notag\\
&\quad+a(\alpha^2 A^+ (u)+\alpha\beta C(u))
+\alpha^2 B^+ (u)
-\ds\int_{\R^{N}}f(x,\alpha u^{+})\alpha u^{+}dx\notag\\
&=\alpha\beta a C(u) +
\alpha \beta^3 b A^- (u)C(u)
+\alpha^2 a A^+ (u)\notag\\
&\quad+2\alpha^{2} \beta^2 b C^2(u)
+\alpha^2 \beta^2 b A^+ (u)A^- (u)
+\alpha^2 B^+ (u)\notag\\
&\ds\quad+3\alpha^{3} \beta b A^+ (u)C(u)
+\alpha^{4} b (A^+ (u))^2 -
\int_{\R^{N}}f(x,\alpha u^{+})\alpha u^{+}dx\ ,\label{eqs3.1}
\end{align}
and
\begin{align}
\langle I'_b(\alpha u^+ + \beta u^-),\beta u^- \rangle
&=b(\beta^2 A^- (u)+2\beta\alpha C(u)+\alpha^2 A^+ (u))(\beta^2 A^- (u)+\beta\alpha C(u))\notag\\
&\quad+a(\beta^2 A^- (u)+\beta\alpha C(u))
+\beta^2 B^- (u)
-\ds\int_{\R^{N}}f(x,\beta u^-)\beta u^- dx\notag\\
&=\alpha\beta a C(u)
+  \alpha^{3}\beta b A^+ (u)C(u)
+ \beta ^2 a A^- (u)\notag\\
&\quad+2  \alpha^2\beta^2  b C^2(u)
+   \alpha^2\beta^2 b A^- (u)A^+ (u)
+\beta^2 B^- (u)\notag\\
&\quad+3 \alpha\beta^3 b A^- (u) C(u)
+\beta^{4}b(A^- (u))^2
-\int_{\R^{N}}f(x,\beta u^-)\beta u^- dx.\label{eqs3.2}
\end{align}
We will show that
there exists $r\in(0,R)$ such that
\begin{equation}\label{eqs3.3}
\langle I'_b(r u^+ + \beta u^-),r u^+ \rangle
>0,
\quad\langle I'_b(\alpha u^+ + r u^-),r u^- \rangle
>0,
\ \ \forall \alpha,\beta\in[r,R],
\end{equation}
and
\begin{equation}\label{eqs3.4}
\langle I'_b(R u^+ + \beta u^-),R u^+ \rangle
<0,
\quad\langle I'_b(\alpha u^+ + R u^-),R u^- \rangle
<0,
\ \ \forall\alpha,\beta\in[r,R].
\end{equation}
Indeed, it follows from the assumption $(f_1)$ that for any $\varepsilon >0$, there exists a positive constant $C_\varepsilon$ such that
\begin{equation}\label{eqs3.5}
\int_{\R^{N}}f(x,\alpha u^{+})\alpha u^{+}dx
\le
\varepsilon\int_{\R^{N}}|\alpha u^{+}|^{2}dx + C_\varepsilon \int_{\R^{N}}|\alpha u^{+}|^{2^{*}_s}dx.
\end{equation}
Choosing $\varepsilon=\frac{1}{2}aA^+ (u)$,
together with (\ref{eqs3.1}) and (\ref{eqs3.5}), we deduce that
\begin{align}
\langle I'_b(\alpha u^+ + \beta u^-),\alpha u^+ \rangle
&\ds\ge\alpha\beta a C(u) +
\alpha \beta^3 b A^- (u)C(u)
+2\alpha^{2} \beta^2 b C^2(u)\notag\\
&\quad
+\alpha^2 \beta^2 b A^+ (u)A^- (u)
+\alpha^2 B(u)+\alpha^{4} b (A^+ (u))^2\notag\\
&\ds\quad+3\alpha^{3} \beta b A^+ (u)C(u)
+\frac{\alpha^2 a}{2}(A^+ (u))^2
- C_1 \int_{\R^{N}}|\alpha u^{+}|^{2^{*}_s}dx,\label{eqs3.6}
\end{align}
where $C_1$ is a positive constant. On the other hand,
since $u^{+} \neq 0$, there exists a constant $\delta>0$ such that $meas\{x\in\R^{N}:u^{+}(x)>\delta\}>$0. In addition, by $(f_3)$ and $(f_4)$, we conclude that for any $L>0$, there exists $T>0$ such that $f(x,s)/s^3>L$ for all $s>T$.
Therefore, for $\alpha > T/\delta$, we have
\begin{equation}\label{eqs3.7}
\int_{\R^{N}}f(x,\alpha u^{+})\alpha u^{+}dx
\ge
\int_{\{u^{+}(x)>\delta\}}
\frac{f(x,\alpha u^{+})}{(\alpha u^{+})^3}(\alpha u^{+})^4
\ge L \alpha^4 \int_{\{u^{+}(x)>\delta\}}
(u^+)^4dx.
\end{equation}
Choose $L$ sufficiently large 
so that $L\int_{\{u^{+}(x)>\delta\}}(u^+)^4dx
>2b(A^+ (u))$. By (\ref{eqs3.1}) and (\ref{eqs3.7}),
we have that
\begin{align}
\langle I'_b(\alpha u^+ + \beta u^-),\alpha u^+ \rangle
&\le\alpha\beta a C(u) +
\alpha \beta^3 b A^- (u)C(u)
+\alpha^2 a A^+ (u)\notag\\
&\quad+2\alpha^{2} \beta^2 b C^2(u)
+\alpha^2 \beta^2 b A^+ (u)A^- (u)
+\alpha^2 B^+ (u)\notag\\
&\ds\quad+3\alpha^{3} \beta b A^+ (u)C(u)
-\alpha^{4} b (A^+ (u))^2 .\label{eqs3.8}
\end{align}
Similarly, we can obtain
\begin{align}
\langle I'_b(\alpha u^+ + \beta u^-),\beta u^- \rangle
&\ge \alpha\beta a C(u)
+  \alpha^{3}\beta b A^+ (u)C(u)
+2  \alpha^2\beta^2  b C^2(u)
\notag\\
&\quad
+\alpha^2\beta^2 b A^- (u)A^+ (u)
+\beta^2 B^- (u)
+\beta^{4}b(A^- (u))^2\notag\\
&\quad+3 \alpha\beta^3 b A^- (u) C(u)
+\frac{\beta^2 b}{2}(A^- (u))^2
-C_2 \int_{\R^N}|\beta u^-|^{2^{*}_s}dx,
\label{eqs3.9}
\end{align}
 and
\begin{align}
\langle I'_b(\alpha u^+ + \beta u^-),\beta u^- \rangle
&\le \alpha\beta a C(u)
+  \alpha^{3}\beta b A^+ (u)C(u)
+\beta^2 a A^- (u)
\notag\\
&\quad
+2  \alpha^2\beta^2  b C^2(u)
+\alpha^2\beta^2 b A^- (u)A^+ (u)
+\beta^2 B^- (u)
\notag\\
&\quad+3 \alpha\beta^3 b A^- (u) C(u)
-\beta^4 b(A^- (u))^2\ ,\label{eqs3.10}
\end{align}
where $C_2$ is a positive constant.
Hence, in view of (\ref{eqs3.6}) and (\ref{eqs3.8})-(\ref{eqs3.10}), we have that there exists $r\in(0,R)$ such that
(\ref{eqs3.3}) and (\ref{eqs3.4}) hold.
It follows, in view of Lemma \ref{lm2.3}, that there exists $(\alpha_u,\beta_u)\in [r,R]\times[r,R]$ which satisfies $W(\alpha_u,\beta_u)=(0,0)$, i.e.,
$\alpha_u u^+ +\beta_u u^-\in \mathcal{N}^b_{nod}. $

It remains to establish the uniqueness of the pair ($\alpha_u,\beta_u$) and we need to consider two cases.

{\bf Case 1.\ $u\in \mathcal{N}^b_{nod}.$}

If $u\in \mathcal{N}^b_{nod}$, then $u^{+}+u^-=u\in \mathcal{N}^b_{nod}$. This in turn implies that
\begin{equation}\label{eqs3.11}
a(A^+ (u)+C(u))+b(A^+ (u)+2C(u)+A^- (u))(A^+ (u)+C(u))
+B^+ (u)=\int_{\R^{N}}f(x,u^+)u^+dx\ ,
\end{equation}
and
\begin{equation}\label{eqs3.12}
a(A^-(u)+C(u))+b(A^- (u)+2C(u)+A^+ (u))(A^- (u)+C(u))
+B^- (u)=\int_{\R^{N}}f(x,u^-)u^-dx.
\end{equation}
We claim that $(\alpha_u,\beta_u)=(1,1)$ is the unique pair of numbers such that $\alpha_u u^+ +\beta_u u^-\in \mathcal{N}^b_{nod}.$

In fact, let $(\tilde{\alpha}_u,\tilde{\beta}_u)$ be a pair of numbers such that $\tilde{\alpha}_u u^+ + \tilde{\beta}_u u^-\in \mathcal{N}^b_{nod}$. By (\ref{eqs3.1}) and (\ref{eqs3.2}), we have
\begin{align}
b(\tilde{\alpha}_u ^2 A^+ (u)+2\tilde{\alpha}_u \tilde{\beta}_u C(u)+\tilde{\beta}_u^2 &A^- (u))(\tilde{\alpha}_u^2 A^+ (u)+\tilde{\alpha}_u\tilde{\beta}_u C(u))
+\tilde{\alpha}_u^2 B(u) \notag\\
+a&(\tilde{\alpha}_u^2 A^+ (u)+\tilde{\alpha}_u\tilde{\beta}_u C(u))
=\ds\int_{\R^{N}}f(x,\tilde{\alpha}_u u^{+})\tilde{\alpha}_u u^{+}dx\ , \label{eqs3.13}
\end{align}
and
\begin{align}
b(\tilde{\beta}_u^2 A^- (u)+2\tilde{\beta}_u\tilde{\alpha}_u C(u)+\tilde{\alpha}_u^2 &A^+ (u))(\tilde{\beta}_u A^- (u)+\tilde{\beta}_u\tilde{\alpha}_u C(u))
+\tilde{\beta}_u^2 B^- (u) \notag\\
+a&(\tilde{\beta}_u^2 A^- (u)+\tilde{\beta}_u\tilde{\alpha}_u C(u))
=\ds\int_{\R^{N}}f(x,\tilde{\beta}_u u^{-})\tilde{\beta}_u u^{-}dx. \label{eqs3.14}
\end{align}
Without loss of generality, we may assume that $0<\tilde{\alpha}_u\le \tilde{\beta}_u$. Then (\ref{eqs3.13}) leads to that
\begin{align}
b\tilde{\alpha}_u^4 (A^+ (u)+2 C(u)+ A^- (u))( A^+ (u)+C(u))
+\tilde{\alpha}_u^2 (aA^+ (u)+ a&C(u)+B^+ (u))\notag\\
\le\ds\int_{\R^{N}}&f(x,\tilde{\alpha}_u u^{+})\tilde{\alpha}_u u^{+}dx. \label{eqs3.15}
\end{align}
From (\ref{eqs3.11}) and (\ref{eqs3.15}) we see that
\begin{equation}\label{eqs3.16}
(\tilde{\alpha}_u{-2}-1)(aA^+ (u)+aC(u)+B^+(u))
\le
\int_{\R^{N}} \Big(\frac{f(x,\tilde{\alpha}_u u^{+})}{(\tilde{\alpha}_u u^{+})^{3}}
-\frac{f(x,u^{+})}{(u^{+})^{3}}
\Big)(u^{+})^4 dx.
\end{equation}
By $(f_{4})$ together with (\ref{eqs3.16}), it implies that
$1\le\ \tilde{\alpha}_u \le \tilde{\beta_u}$. Using the same method, we can get $\tilde{\beta}_u \le 1$ by (\ref{eqs3.12}) and (\ref{eqs3.14}), which implies
$\tilde{\alpha}_u = \tilde{\beta}_u=1$.

{\bf Case 2.\ $u\notin \mathcal{N}^b_{nod}.$}

If $u \notin \mathcal{N}^b_{nod}$, there exists a pair $(\alpha_u,\beta_u)$ of positive numbers such that $\alpha_u u^+ + \beta_u u^-\in \mathcal{N}^b_{nod}$. Suppose that there exists another pair $(\alpha'_u,\beta'_u)$ of positive numbers such that $\alpha'_u u^+ + \beta_u' u^-\in \mathcal{N}^b_{nod}$.
Set $v:=\alpha_u u^+ + \beta_u u^-$ and
$v':=\alpha'_u u^+ + \beta'_u u^-$, we have
\begin{equation*}
\ds \frac{\alpha'_u}{\alpha_u} v^{+}
+\frac{\beta'_u}{\beta_u} v^{-}=
\alpha'_u u^{+}+\beta'_u u^{-}=
v' \in \mathcal{N}^b_{nod}.
\end{equation*}
Noticing that $v \in \mathcal{N}^b_{nod}$, we obtain that $\alpha'_u=\alpha_u$ and
$\beta'_u=\beta_u$, which implies that  $(\alpha_u,\beta_u)$ is the unique pair of numbers such that $\alpha_u u^+ +\beta_u u^-\in \mathcal{N}^b_{nod}$. Hence we finish our proof. 
\end{proof}

\begin{lemma}\label{lm3.2}
Assume $(f_1)-(f_4)$ hold, and $u\in H$ such that $\langle I'_b(u),u^{+} \rangle \le 0$ and $\langle I'_b(u),u^{-} \rangle \le 0$. Then the unique pair $(\alpha_u,\beta_u)$ obtained in Lemma \ref{lm3.1} satisfies $0<\alpha_u,\beta_u \le 1$.
\end{lemma}

\begin{proof}
Without loss of generality, we may assume that $\alpha_u \ge \beta_u >0$. Since $\alpha_u u^{+} + \beta_u u^{-}
\in \mathcal{N}^b_{nod}$, we have
\begin{align}
b\alpha_u^4 (A^+ (u)+2 C(u)+& A^- (u))( A^+ (u)+C(u))
+ \alpha_u^2 (aA^+ (u)+ aC(u)+B^+ (u))\notag\\
\ge&b(\alpha_u^2 A^+ (u)+2\alpha_u\beta_u C(u)+\beta_u^2 A^- (u))(\alpha_u^2 A^+ (u)+\alpha_u\beta_u C(u))\notag\\
&+(a\alpha_u^2 A^+ (u)+a\alpha_u\beta_u C(u)+\alpha_u B^+ (u)) \notag \\
=&\ds\int_{\R^{N}}f(x,\alpha_0 u^{+})\alpha_0 u^{+}dx. \label{eqs3.17}
\end{align}
Sine $\langle I'_b(u),u^{-} \rangle \le 0$, it yields that
\begin{equation}\label{eqs3.18}
b(A^+ (u)+2C(u)+A^- (u))(A^+ (u)+C(u))+a(A^+ (u)+C(u))
+B^+ (u) \le \int_{\R^{N}}f(u^+)u^+dx.
\end{equation}
Then, it follows from (\ref{eqs3.17}) and (\ref{eqs3.18}) that
\begin{equation}\label{eqs3.19}
(\alpha_u^{-2}-1)(aA^+ (u)+aC(u)+B^+ (u))
\le
\int_{\R^{N}} \Big(\frac{f(\alpha_u u^{+})}{(\alpha_u u^{+})^{3}}
-\frac{f(u^{+})}{(u^{+})^{3}}
\Big)(u^{+})^4 dx.
\end{equation}
If $\alpha_u>1$, the left side of (\ref{eqs3.19}) is negative while, from $(f_4)$, the right side of \eqref{eqs3.19} is positive. This implies that $\alpha_u \le 1$, which completes the proof.
\end{proof}

\begin{lemma}\label{lm3.3}
For any fixed $u\in H$ with $u^ {\pm} \neq 0$, $(\alpha_u,\beta_u)$ is the unique maximum point of the function $\phi:({\R_+}\times {\R_+})\rightarrow \R$, where $(\alpha_u,\beta_u)$ is obtained in Lemma \ref{lm3.1}, and $\phi (\alpha,\beta):=I_b({\alpha u^+ +\beta u^-})$.
\end{lemma}

\begin{proof}
From the proof of Lemma \ref{lm3.1}, $(\alpha_u,\beta_u)$ is the unique critical point of $\phi$ in $(\R_+\times\R_+)$.
By $(f_{3})$, we conclude that $\phi (\alpha,\beta) \rightarrow -\infty$ uniformly as $|(\alpha,\beta)|\rightarrow \infty$. Therefore, it is sufficient to check that the maximum point cannot be achieved on the boundary of $(\R_+\times\R_+)$. We use the argument of contradiction. 
Suppose that $(0,\bar{\beta})$ is the global maximum point of $\phi$ with $\bar{\beta}\ge 0$.
By a direct calculation, we have
\begin{align}
\phi (\alpha,\bar{\beta})=&I_b(\alpha u^{+}+\bar{\beta}u^{-})\notag\\
=&\ds\frac{a}{2} \int_{\R^{N}} |(-\Delta)^{s/2}(\alpha u^+ + \bar{\beta} u^-)|^2dx+\int_{\R^{N}}V(x)(\alpha u^+ + \bar{\beta} u^-)^2 dx\notag\\
&\ds+\frac{b}{4}\Big(\int_{\R^{N}} |(-\Delta)^{s/2}(\alpha u^+ + \bar{\beta} u^-)|^2dx\Big)^2-
\int_{\R^{N}} F(x,\alpha u^+ +
\bar{\beta} u^-)dx, \notag
\end{align}
which implies that $\phi$ is an increasing function with respect to $\alpha$ if $\alpha$ is small enough.
This yields the contradiction. Similarly, $\phi$ can not achieve its global maximum on $(\alpha,0)$
for any $\alpha\ge 0.$ We finish the proof.
\end{proof}
Next we consider the following minimization problem
\begin{equation}\label{eqs3.20}
c_{nod}^{b}:=\inf\{I_b(u):u\in \mathcal{N}^b_{nod}\}.
\end{equation}
As $\mathcal{N}^b_{nod}$ is nonempty in $H$ by Lemma \ref{lm2.1}, we see that $c_{nod}^{b}$ is well defined.

\begin{lemma}\label{lm3.4}
Assume that $(f_1)-(f_4)$ hold, then $c_{nod}^{b}$ can be achieved.
\end{lemma}
\begin{proof}
For every $u\in \mathcal{N}^b_{nod}$, we have
$\langle I'_b(u),u\rangle=0$.
Then by $(f_1),(f_2)$ and Lemma \ref{lm2.1}, we get
\begin{align}
a\|u\|^{2}_H &\le
a\int_{\R^{N}} |(-\Delta)^{s/2}u|^2dx+\int_{\R^{N}}V(x)|u|^2 dx+b\Big(\int_{\R^{N}}
|(-\Delta)^{s/2}u|^2dx\Big)^2 \notag\\
&=\int_{\R^{N}}f(x,u)u dx\notag\\
&\le \varepsilon \int_{\R^{N}}|u|^2dx
+C_\varepsilon \int_{\R^{N}}|u|^{p}dx,\label{eqs3.21}
\end{align}
where $C_{\varepsilon}$ is a positive constant depending only on $\varepsilon$.
Choosing $\varepsilon$ sufficiently small such that
$\ds \varepsilon\|u\|^{2}_{L^{2}(\R^N)} \le
\frac{a}{2}\|u\|^{2}_H$, we can then deduce
that there exists a constant $\mu>0$ such that $\|u\|^{2}_H\ge \mu$. And by $(f_4)$, for $s\neq0$, we have
\begin{equation*}
f(x,s)-4F(x,s)\ge 0.
\end{equation*}
Hence
\begin{equation}\label{eqs3.22}
\ds I_b(u)=I_b(u)-\frac{1}{4}
\langle I'_b(u),u\rangle
\ge \frac{a}{4} \|u\|^{2}_H
\ge \frac{a\mu}{4} ,
\end{equation}
which implies that $c_{nod}^{b}\ge\frac{1}{4} \mu>0.$

Let $\{u_n\}\subset \mathcal{N}^b_{nod}$ be a minimizing sequence such that $I_b(u_n)
\rightarrow c_{nod}^{b}.$ By (\ref{eqs3.22}), we have $\{u_n\}$ is bounded in $H$. Utilizing Lemma \ref{lm2.1} and the properties of $L^p$-space, up to a subsequence, we have
\begin{align}
&u_{n}^{\pm} \rightharpoonup u_{b}^{\pm}\ \  \mathrm{weakly} \ \mathrm{in}\ H,\qquad
u_{n}^{\pm} \rightarrow u_{b}^{\pm}\ \ \mathrm{a.e.}\ \mathrm{in}\ \R^{N},\notag\\
&u_{n}^{\pm} \rightarrow u_{b}^{\pm}\ \  \mathrm{strongly} \ \mathrm{in}\ L^q({\R^N}),\quad\text{for}\ \ q\in[2,2^*_s), \label{eqs3.23}\\
&(-\Delta)^{s/2}u_{n}^{\pm} \rightarrow
(-\Delta)^{s/2}u_{b}^{\pm}\ \ \mathrm{a.e.}\ \mathrm{in}\ \R^{N}.\notag
\end{align}
Furthermore, the conditions $(f_1)-(f_2)$ combined with the compactness lemma of Strauss \cite{SW} gives that
\begin{equation}\label{eqs3.24}
\lim\limits_{n\rightarrow\infty}\int_{\R^{N}}f(x,u_{n}^{\pm})u_{n}^{\pm}dx=
\int_{\R^{N}}f(x,u_{b}^{\pm})u_{b}^{\pm}dx,
\ \lim\limits_{n\rightarrow\infty}\int_{\R^{N}}F(x,u_{n}^{\pm})dx=
\int_{\R^{N}}F(x,u_{b}^{\pm})dx.
\end{equation}
Since ${u_n}\in \mathcal{N}^b_{nod}$, we have $\langle I'_{b}(u_n),u_n^{\pm}\rangle=0$, that is,
\begin{align}
a&\Big(   \int_{\R^{N}} |(-\Delta)^{s/2}(u_n^+)|^2dx +   \int_{\R^{N}} (-\Delta)^{s/2}(u_n^{+})
(-\Delta)^{s/2}(u_n^{-})dx
\Big)+\int_{\R^{N}} V(x)|u_n^{+}|^2dx\notag\\
&\ \ +b\Big(\int_{\R^{N}} |(-\Delta)^{s/2}u_n|^2dx\Big)
\Big(\int_{\R^{N}} |(-\Delta)^{s/2}(u_n^+)|^2dx +   \int_{\R^{N}} (-\Delta)^{s/2}(u_n^{+})
(-\Delta)^{s/2}(u_n^{-})dx
\Big)\notag\\
&=\int_{\R^{N}}f(x,u_n^{+})u_n^{+}dx\ ,\label{eqs3.25}
\end{align}
and
\begin{align}
a&\Big(   \int_{\R^{N}} |(-\Delta)^{s/2}(u_n^-)|^2dx +   \int_{\R^{N}} (-\Delta)^{s/2}(u_n^{-})
(-\Delta)^{s/2}(u_n^{+})dx
\Big)+\int_{\R^{N}} V(x)|u_n^{+}|^2dx\notag\\
&\ \ +b\Big(\int_{\R^{N}} |(-\Delta)^{s/2}u_n|^2dx\Big)
\Big(\int_{\R^{N}}|(-\Delta)^{s/2}(u_n^-)|^2dx + \int_{\R^{N}} (-\Delta)^{s/2}(u_n^{-})
(-\Delta)^{s/2}(u_n^{+})dx
\Big)\notag\\
&=\int_{\R^{N}}f(u_n^{+})u_n^{+}dx.\label{eqs3.26}
\end{align}
With a similar argument as (\ref{eqs3.21}), there exists a constant $l>0$ such that $\|u_n^{\pm}\|^2_H\ge l$ for all $n\in\N$.
From (\ref{eqs3.25}), we obtain
that $\int_{\R^{N}}f(x,u_n^{\pm})u_n^{\pm}\ge l$. Hence by (\ref{eqs3.24}), we have $\int_{\R^{N}}f(u_b^{\pm})u_b^{\pm}\ge l$, i.e. $u_b^{\pm}\neq 0$.

From Lemma \ref{lm3.1}, there exist $\alpha_{u_b},\beta_{u_b}>0$ such that
$\ds \bar{u}_b:=\alpha_{u_b} u_b^+ + \beta_{u_b} u_b^- \in
\mathcal{N}^b_{nod},$  that is,
$$
\langle I'_b(\alpha_{u_b} u_b^+ + \beta_{u_b} u_b^-),\alpha_{u_b} u_b^+ \rangle=
\langle I'_b(\alpha_{u_b} u_b^+ + \beta_{u_b} u_b^-),\beta u_b^-\rangle=0.
$$
Now, we show that $\alpha_{u_b},\beta_{u_b} \le 1$. In fact,
by (\ref{eqs3.23})-(\ref{eqs3.25}) and Fatou's Lemma, we conclude that
\begin{align}
a&\Big(   \int_{\R^{N}} |(-\Delta)^{s/2}(u^+)|^2dx +   \int_{\R^{N}} (-\Delta)^{s/2}(u^{+})
(-\Delta)^{s/2}(u^{-})dx
\Big)+\int_{\R^{N}} V(x)|u^{+}|^2dx\notag\\
&\ \  +b\Big(\int_{\R^{N}} |(-\Delta)^{s/2}u|^2dx\Big)
\Big(\int_{\R^{N}} |(-\Delta)^{s/2}(u^+)|^2dx +   \int_{\R^{N}} (-\Delta)^{s/2}(u^{+})
(-\Delta)^{s/2}(u^{-})dx
\Big)\notag\\
&\le\int_{\R^{N}}f(x,u^{+})u^{+}dx.\label{eqs3.27}
\end{align}
It follows from Lemma \ref{lm2.2} and (\ref{eqs3.27}) that
$\alpha_{u_b}\le 1$. Similarly, $\beta_{u_b}\le 1$.
It follows from $(f_4)$ that $H(s):=sf(x,s)-4F(x,s)$ is a non-negative function, which is also increasing in $|s|$. Hence we have
{\allowdisplaybreaks
\begin{align}
c_{nod}^{b} \le& I(\bar{u}_b)=I(\bar{u}_b)
-\frac{1}{4}\langle I'(\bar{u}_b),\bar{u}_b\rangle\notag\\
=&\frac{a}{4}\|\bar{u}_b\|_H^2+
\frac{1}{4}\int_{\R^N}\Big(
f(\bar{u}_b)\bar{u}_b
-4F(\bar{u}_b)\Big)dx\notag\\
=&\frac{a\alpha^2_{u_b}}{4}\| u^+_b\|_H^2+
\frac{a\alpha_{u_b}\beta_{u_b}}{2}
\int_{\R^N} (-\Delta)^{s/2}(u_b^{+})
(-\Delta)^{s/2}(u_b^{-})dx
+ \frac{a\beta^2_{u_b}}{4}
\| u^-_b\|_H^2  \notag\\
&+ \frac{1}{4}\int_{\R^N}\Big(
f(\alpha_{u_b}u^+_b)\alpha_{u_b}u^+_b
-4F(\alpha_{u_b}u^+_b)\Big)dx
+ \frac{1}{4}\int_{\R^N}\Big(
f(\beta_{u_b}u^-_b)\beta_{u_b}u^-_b
-4F(\beta_{u_b}u^-_b)\Big)dx\notag\\
\le&\frac{a}{4}\|u_b\|_H^2+
\frac{1}{4}\int_{\R^N}\Big(
f(u_b)u_b
-4F(u_b)\Big)dx\notag\\
\le&\liminf\limits_{n\rightarrow\infty}
\Big[I_b (u_n)-\frac{1}{4}
\langle I'_b (u_n), u_n \rangle
\Big]=c^b_{nod}\ .\notag
\end{align}}
From the analysis above, we conclude that $\alpha_{u_b}=\beta_{u_b}=1$. Thus,
$\bar{u}_b=u_b$ and $I_b(u_b)=c^b_{nod}$.
\end{proof}

\section{\label{sec4} Proofs of main theorems}

This section is devoted to prove our main results. We first prove that the minimizer $u_b$ for the problem
(\ref{eqs3.20}) is indeed a sign-changing solution of (1.1), that is, $I'_b(u_b)=0$.
\begin{proof}[\bf Proof of Theorem 1.1.]
Applying the quantitative deformation lemma (see \cite{WM}), we want to prove that $I'_b(u_b)=0$.
We first recall that $\langle I'_b(u_b),u^{+}_b\rangle=\langle I'_b(u_b),u^{-}_b\rangle=0$. Then it follows from Lemma \ref{lm3.3} that
\begin{equation}\label{eqs4.1}
I_b(\alpha u_b^+ +\beta u_b^-)<I_b(u_b^+ + u_b^-)=
c_{nod}^b
\end{equation}
for $(\alpha,\beta)
\in(\R_+\times\R_+)$ and $(\alpha,\beta)\neq(1,1).$ If $I'_b(u_b)\neq0$, then there exist $\delta>0$ and $\lambda>0$ such that
\begin{equation*}
\|I'_b(u_b)\|\ge \lambda \ \ \ \mathrm{for \ all}\ \|v-u_b\|_H\le 3\delta.
\end{equation*}
Let $D:=(\frac12,\frac32)\times
(\frac12,\frac32)$ and $g(\alpha,\beta):=
\alpha u_b^+ +\beta u_b^-$. In view of lemma \ref{lm3.3} again, we have
\begin{equation}\label{eqs4.2}
m^b:=\max\limits_{\partial D}
I_b\circ g <c_{nod}^b.
\end{equation}
Let $\varepsilon=\min\{(c^b_{nod}-m^b)/2,
\lambda\delta/8\}$ and $S=B(u_b,\delta)$. It follows from lemma 2.2 in \cite{WM} that there is a continuous mapping $\eta$ such that
\begin{enumerate}
	\item  [($a$)] \ $\eta(1,u)=u$ if
	$u\neq I_b^{-1}([c^b_{nod}-2\varepsilon,
	c^b_{nod}+2\varepsilon])\cap S_{2\delta}$;
	\item [($b$)] \ $\eta(1,I_b^{c^b_{nod}+\varepsilon}\cap S)
	\subset I_b^{c^b_{nod}-\varepsilon}$;
	\item [($c$)] \ $I_b(\eta(1,u))\le I_b(u)$ for all $u\in H$.
\end{enumerate}
It is clear that
\begin{equation}\label{eqs4.3}
\max\limits_{(\alpha,\beta)\in \overline{D}}
I_b\big( \eta(1,g(\alpha,\beta))
\big)<c^b_{nod}.
\end{equation}

Next, we prove that $\eta(1,g(D))\cap \mathcal{N}^b_{nod}\neq\varnothing,$ which contradicts to the definition of $c^b_{nod}$. Let us define $h(\alpha,\beta):=\eta(1,g(\alpha,\beta))$ and

\begin{align}
\Psi_0(\alpha,\beta):=&
\Big(\langle I'_b\big(g(\alpha,\beta)\big),u^{+}_b\rangle, \langle I'_b\big(g(\alpha,\beta)\big),u^{-}_b\rangle
\Big)\notag \\
=&\Big(
\langle I'_b\big(\alpha u^+_b +
\beta u^{-}_b \big),u^{+}_b\rangle, \langle I'_b\big(\alpha u^+_b +
\beta u^{-}_b \big),u^{-}_b\rangle
\Big),\notag
\end{align}
$$
\Psi_1(\alpha,\beta):=
\Big(\frac{1}{\alpha} \langle I'_b\big(h(\alpha,\beta)\big),h^{+}(\alpha,\beta)\rangle, \frac{1}{\beta}\langle I'_b\big(h(\alpha,\beta)\big),h^{-}(\alpha,\beta)\rangle
\Big).
$$
By Lemma \ref{lm3.1} and the degree theory, this implies that $\deg (\Psi_0,D,0)=1$. It follows, in view of  (\ref{eqs4.2}), that $g=h$ on $\partial D$, from which we obtain
$\deg (\Psi_1,D,0)=\deg (\Psi_0,D,0)=1$.
Therefore, $\Psi_1(\alpha_0,\beta_0)=0$ for some $(\alpha_0,\beta_0)\in D$, so that
$\eta(1,g(\alpha_0,\beta_0))=h(\alpha_0,\beta_0)\in \mathcal{N}^b_{nod}$, which is a contradiction. We have thus proved that $u_b$ is a critical point of $I_b$. Moreover, $u_b$ is a sign-changing solution for problem (1.1).
\end{proof}

By Theorem 1.1, we obtain a least energy sign-changing solution of problem (1.1). Hence,
Theorem 1.2 follows immediately if we establish the strict inequality $c^b_{nod}>
2c^b$, where $c^b=\inf \limits_{u \in    \mathcal{N}^b} I(u)$.

\begin{proof}[\bf Proof of Theorem 1.2.]
Recall that $\mathcal{N}^b$ denotes the Nehari manifold associated to (1.1) and $c^b=\inf \limits_{u \in  \mathcal{N}^b} I(u)$.
Then, by a similar argument to that in the proof of Lemma \ref{lm3.4}, there exists $v_b \in \mathcal{N}^b$ such that
$I_b(v_b)=c^b >0$. We assert that $v_b$ is actually a ground state solution of (1.1).
In fact, by virtue of section 4 in \cite{WM}, $v_b$ is a critical point of
$I_b$ on $H$, or equivalently, $I'(v_b)=0$.

Reviewing Theorem 1.1 again, we know that $u_b$ is a least sign-changing solution of
problem (1.1). Utilizing the same method in Lemma \ref{lm3.1}, we conclude that there exists  $\alpha_{u^{+}_b}>0$ such that
$\alpha_{u^{+}_b} u^{+}_b$ $\in \mathcal{N}^b$. And we can prove that there  exists  $\beta_{u^{-}_b}>0$ such that
$\beta_{u^{-}_b} u^{-}_b$ $\in \mathcal{N}^b$ analogously. Moreover, Lemma \ref{lm3.3} implies that
$\alpha_{u^{+}_b}, \beta_{u^{-}_b} \in
(0,1)$.

Therefore, in view of Lemma \ref{lm3.4}, we obtain that
\begin{equation}\label{eqs4.4}
2c^b \le I_b(\alpha_{u^{+}_b} u^{+}_b)
+ I_b(\beta_{u^{-}_b} u^{-}_b)
\le I_b(\alpha_{u^{+}_b} u^{+}_b +\beta_{u^{-}_b} u^{-}_b)
\le I(u^{+}_b +u^{-}_b)=c^b_{nod}.
\end{equation}
It follows that $c^b>0$ which cannot be achieved by a sign-changing solution. Hence, we complete the proof.
\end{proof}
Finally, we close this section with the proof of Theorem 1.3. In the following, we regard
$b>0$ as a parameter in problem (1.1).

\begin{proof}[\bf Proof of Theorem 1.3.]
We shall proceed through several claims on
analyzing the convergence property of $u_b$ as $b \rightarrow 0$, where $u_b$ is the least energy sign-changing solution obtained in Theorem 1.1.

\noindent{\bf Claim 1.}
For any sequence $\{b_n\}$ as $b_n\searrow 0$, $\{u_{b_n}\}$ is bounded in $H$.
\begin{proof}%[\bf Proof of the claim.]
We select a nonzero function $\psi \in \mathcal{C}_c^{\infty} $ with $\psi ^{\pm} \neq 0$.
By the assumptions $(f_3)$ and $(f_4)$, we deduce that, for any
$b\in [0,1]$, there exists a pair $(\lambda_1,\lambda_2)$ independent of $b$,
such that
\begin{equation*}
\langle I'_b(\lambda_1 \psi^+ + \lambda_2 \psi^-),\lambda_1 \psi^+ \rangle<0
\quad \mathrm{and}\quad
\langle I'_b(\lambda_1 \psi^+ + \lambda_2 \psi^-),\lambda_2 \psi^- \rangle<0.
\end{equation*}
Then by virtue of Lemma \ref{lm2.2}, we get that, for any $b\in [0,1]$, there exists a unique pair $(\alpha_{\psi}(b),\beta_{\psi}(b)) \in (0,1]\times (0,1]$ such that
\begin{equation}\label{eqs4.5}
\overline{\psi}:=\alpha_{\psi}(b)\lambda_1 \psi^+ +\beta_{\psi}(b) \lambda_2 \psi^- \in \mathcal{N}^b_{nod}.
\end{equation}
Recall that $(f_1)$ implies that
$f(s)\le C_1 |s|+C_2 |s|^{p-1}$ and $F(s)\le C_1 |s|^2+C_2 |s|^{p}$, where $C_1$ and $C_2$
are positive constants.
Then it follows that
\begin{align}
I_b(u_b)&\le I_b(\overline{\psi})=I_b(\overline{\psi})-
\langle I'_b(\overline{\psi}),\overline{\psi} \rangle
\notag\\
&=\frac{a}{4}\|\overline{\psi}\|^2_H + \int_{\R^{N}}
\left(f(x,\overline{\psi})\overline{\psi}-4F(x,\overline{\psi})\right)dx\notag\\
&=\frac{a}{4}\| \lambda_1 \psi^+\|_H^2+
\frac{a}{2}
\int_{\R^N} (-\Delta)^{s/2}(\lambda_1 \psi^+)
(-\Delta)^{s/2}(\lambda_2 \psi^-)dx
+ \frac{a}{4}
\| \lambda_2 \psi^-\|_H^2  \notag\\
&\quad + C_1\int_{\R^N}\left(
\lambda^2_1 |\psi^{+}|^2 +
\lambda^2_2 |\psi^-|^2\right)dx
+ C_2\int_{\R^N}\left(
\lambda^p_1 |\psi^{+}|^p +
\lambda^p_2 |\psi^-|^p
\right)dx\notag\\
&:=C_0,\label{eqs4.6}
\end{align}
where $C_0$ is a positive constant independent of $b$. Then we write, as $n\rightarrow \infty$,
\begin{equation}\label{eqs4.7}
C_0 +1\ge I_{b_n}(u_{b_n}) =
I_{b_n}(u_{b_n})-\frac{1}{4}\langle I'_{b_n}(u_{b_n}),u_{b_n} \rangle\ge
\frac{a}{4} \|u_{b_n}\|^2_H,
\end{equation}
from which the claim follows.
\end{proof}

\vskip 0.2cm
\noindent{\bf Claim 2.}
 (\ref{eqs1.22}) possesses one sign-changing solution $u_0$.

\begin{proof}
Going if necessary to a subsequence, thanks to Lemma \ref{lm2.1}, we conclude that there exists $u_0 \in H$, such that
\begin{align}
&u_{b_n} \rightharpoonup u_{0}\ \  \mathrm{weakly} \ \mathrm{in}\ H, \notag\\
&u_{b_n} \rightarrow u_{0}\ \  \mathrm{strongly} \ \mathrm{in}\ L^q({\R^N}) \ \ \text{for}\ \ q\in[2,2^*_s),\label{eqs4.8}\\
&u_{b_n} \rightarrow u_{0}\ \ \mathrm{a.e.}\ \mathrm{in}\ \R^{N}. \notag
\end{align}
Since $u_{b_n}$ is a weak solution of (1.1) with $b=b_n$, we then have
\begin{multline}
a \int_{\R^N}(-\Delta)^{s/2}u_{b_n}(-\Delta)^{s/2}\phi dx\\
+b_n\int_{\R^N}|(-\Delta)^{s/2}u_{b_n}|^2dx
\int_{\R^N}(-\Delta)^{s/2}u_{b_n}(-\Delta)^{s/2}\phi dx
+\int_{\R^N}V(x)u_{b_n}^2 dx\\
=\int_{\R^N}f(x,u_{b_n})\phi dx,\label{eqs4.9}
\end{multline}
for all $\phi \in \mathcal{C}_c^{\infty}(\R^N)$.
From (\ref{eqs3.24}), (\ref{eqs4.8}), (\ref{eqs4.9}) and Claim 1, we see that
\begin{equation}\label{eqs4.10}
a \int_{\R^N}(-\Delta)^{s/2}u_{0}(-\Delta)^{s/2}\phi dx
+\int_{\R^N}V(x)u_{0}^2 dx=\int_{\R^N}f(x,u_0)\phi dx\ 
\end{equation}
for all $\phi \in \mathcal{C}_c^{\infty}(\R^N)$,
which yields that $u_0$ is a weak solution of (\ref{eqs1.22}). It remains now to establish that $u_0^{\pm}\neq 0.$ Indeed, via a similar argument to that in the proof in Lemma \ref{lm3.4}, we conclude that
$\ds\int_{\R^N}f(u_0^{\pm})u_0^{\pm}>c>0$, where $c$ is a positive constant. Therefore, we complete the proof of the claim.
\end{proof}

\vskip 0.2cm

\noindent{\bf Claim 3.}
(\ref{eqs1.22}) possesses a least energy sign-changing solution $v_0$, and there exists a unique pair $(\alpha_{b_n},\beta_{b_n}) \in \R_+\times\R_+$  such that $\alpha_{b_n}v^{+}_0+\beta_{b_n}v^{-}_0
\in \mathcal{N}^{b_n}_{nod}$. Moreover,
($\alpha_{b_n},\beta_{b_n}) \rightarrow (1,1)$ as $n \rightarrow \infty.$

\begin{proof}
With a similar argument to the proof of Theorem 1.1, we have that (\ref{eqs1.22}) possesses a least energy sign-changing solution $v_0$, where $I(v_0) = c_{nod}^0$ and $I'(v_0)=0$.
Then, by Lemma \ref{lm2.1}, we can easily obtain the existence and uniqueness of the pair ($\alpha_{b_n},\beta_{b_n}$) such that
$\alpha_{b_n}v^{+}_0+\beta_{b_n}v^{-}_0
\in \mathcal{N}^{b_n}_{nod}$. Besides,
we have $\alpha_{b_n},\beta_{b_n}>0.$
Then the claim will follow once we have
proved that $(\alpha_{b_n},\beta_{b_n})\rightarrow (1,1)$ as $n \rightarrow \infty.$
In fact, since $\alpha_{b_n}v^{+}_0+\beta_{b_n}v^{-}_0
\in \mathcal{N}^{b_n}_{nod}$, we then have
\begin{align}
a&\Big(\alpha^2_{b_n}\int_{\R^N}|(-\Delta)^{s/2}(v^+_0)|^2dx+\alpha_{b_n}\beta_{b_n}\int_{\R^N}(-\Delta)^{s/2}(v^+_0)(-\Delta)^{s/2}(v^-_0)dx
\Big)\notag\\
&\quad+b_n
\left(\alpha^2_{b_n}\int_{\R^N}|(-\Delta)^{s/2}(v^-_0)|^2dx+\alpha_{b_n}\beta_{b_n}\int_{\R^N}(-\Delta)^{s/2}(v^+_0)(-\Delta)^{s/2}(v^-_0)dx\right)\notag\\
&\ \ \times\Big(\int_{\R^N}|(-\Delta)^{s/2}(\alpha_{b_n}v^{+}_0+\beta_{b_n}v^{-}_0)|^2dx\Big)
+\alpha^2_{b_n}\int_{\R^N}V(x)|v^+_0|^2dx\notag\\
&=\int_{\R^N}f(x,\alpha_{b_n}v^{+}_0)\alpha_{b_n}v^{+}_0 dx\ ,\label{eqs4.11}
\end{align}
and
\begin{align}
a&\Big(\beta^2_{b_n}\int_{\R^N}|(-\Delta)^{s/2}(v^-_0)|^2dx+\alpha_{b_n}\beta_{b_n}\int_{\R^N}(-\Delta)^{s/2}(v^+_0)(-\Delta)^{s/2}(v^-_0)dx
\Big)\notag\\
&\quad+b_n
\Big(\beta^2_{b_n}\int_{\R^N}|(-\Delta)^{s/2}(v^+_0)|^2dx+\alpha_{b_n}\beta_{b_n}\int_{\R^N}(-\Delta)^{s/2}(v^+_0)(-\Delta)^{s/2}(v^-_0)dx\Big)\notag\\
&\ \  \times\Big(\int_{\R^N}|(-\Delta)^{s/2}(\alpha_{b_n}v^{+}_0+\beta_{b_n}v^{-}_0)|^2dx\Big)
+\alpha^2_{b_n}\int_{\R^N}V(x)|v^+_0|^2dx\notag\\
&=\int_{\R^N}f(x,\beta_{b_n}v^{-}_0)\beta_{b_n}v^{-}_0 dx.\label{eqs4.12}
\end{align}
From $(f_3)$-$(f_4)$, and recalling that $b_n \rightarrow 0$ as $n \rightarrow
\infty,$ we deduce that $\{\alpha_{b_n}\}$ and $\{\beta_{b_n}\}$ are bounded sequences.
Up to a subsequence, suppose that $\alpha_{b_n} \rightarrow \alpha_0$ and $\beta_{b_n} \rightarrow \beta_0$, then it follows from (\ref{eqs4.11}) and (\ref{eqs4.12}) that
\begin{align}
a \Big(\alpha^2_0 \int_{\R^N}|(-\Delta)^{s/2}(v^+_0)|^2dx
+&\alpha_{0}\beta_{0}\int_{\R^N}(-\Delta)^{s/2}(v^+_0)(-\Delta)^{s/2}(v^-_0)dx
\Big)\notag\\
&+\alpha^2_0 \int_{\R^N} V(x)|v^+_0|^2dx
=\int_{\R^N}
f(x,\alpha_0 v^+_0)\alpha_0 v^+_0dx\ ,\label{eqs4.13}
\end{align}
and
\begin{align}
a \Big(\beta^2_0 \int_{\R^N}|(-\Delta)^{s/2}(v^-_0)|^2dx
+&\alpha_{0}\beta_{0}\int_{\R^N}(-\Delta)^{s/2}(v^+_0)(-\Delta)^{s/2}(v^-_0)dx
\Big)\notag\\
&+\beta^2_0 \int_{\R^N} V(x)|v^-_0|^2dx
=\int_{\R^N}
f(x,\beta_0 v^-_0)\beta_0 v^-_0dx.\label{eqs4.14}
\end{align}
Noticing that $v_0$ is a solution of (\ref{eqs1.22}), we then have
\begin{align}
a \Big(\int_{\R^N}|(-\Delta)^{s/2}(v^+_0)|^2dx
+\int_{\R^N}(-\Delta)&^{s/2}(v^+_0)(-\Delta)^{s/2}(v^-_0)dx
\Big)\notag\\
&+ \int_{\R^N} V(x)|v^+_0|^2dx
=\int_{\R^N}
f(x, v^+_0) v^+_0dx,\label{eqs4.15}
\end{align}
and
\begin{align}
a \Big(\int_{\R^N}|(-\Delta)^{s/2}(v^-_0)|^2dx
+\int_{\R^N}(-\Delta)&^{s/2}(v^+_0)(-\Delta)^{s/2}(v^-_0)dx
\Big)\notag\\
&+ \int_{\R^N} V(x)|v^-_0|^2dx
=\int_{\R^N}
f(x, v^-_0) v^-_0dx.\label{eqs4.16}
\end{align}
Moreover, by the assumptions ($f_3$) and ($f_4$), we conclude that $f(s)/|s|^3$ is nondecreasing in $|s|$. Thus, in view of (\ref{eqs4.13})-(\ref{eqs4.16}), we can easily check that
$(\alpha_0,\beta_0)=(1,1)$, and the claim follows.
\end{proof}

We now come back the proof of Theorem 1.3. We assert that $u_0$ obtained in Claim 2 is a least energy solution of problem (1.2). In fact, by virtue of Claim 3 and Lemma \ref{lm2.3}, we have
\begin{equation*}
\ds I_0(v_0)\le I_0(u_0)= \lim\limits_{n\rightarrow \infty}I_{b_n}(u_{b_n})
\le \lim\limits_{n\rightarrow \infty}I_{b_n}(\alpha_{b_n}v^{+}_0+\beta_{b_n}v^{-}_0)
= \lim\limits_{n\rightarrow \infty}I_{0}(v^{+}_0+v^{-}_0)=I_0(v_0),
\end{equation*}
which yields Theorem \ref{th1.3}.
\end{proof}

{\bf Acknowledgment}

The authors would like to thank Prof. Yinbin Deng for many inspiring discussions with him and his constant encouragement and support, without which this paper would not be possible.


\begin{thebibliography}{99}
	
	\bibitem {ACM}
	C. O. Alves, F. J. S. A. Corr\^ea, T. F. Ma,
	Positive solutions for a quasilinear elliptic equation of Kirchhoff type,
	{\it Comput. Math. Appl.,}
	{\bf 49} (2005), 85--93.
	
	\bibitem {AN}
	C. O. Alves, A. B. N\'{o}brega,
	Nodal ground state solution to a biharmonic equation via dual method.
	{\it J. Differ. Equ.,}
	{\bf 260} (2016), 5174-5201.
	
	
	\bibitem {AR}
	A. Ambrosetti, P. H. Rabinowitz,
	Dual variational methods in critical point theory and applications,
	{\it J. Funct. Anal.,}
	{\bf 14} (1973), 349--381.
	
	
	
%	\bibitem {AFP}
%	G. Autuori, A. Fiscella, P. Pucci,
%	Stationary Kirchhoff problems involving a fractional operator and a critical nonlinearity,
%	{\it Nonlinear Anal.,}
%	{\bf 125} (2015), 699-714.
	

    \bibitem{BCPS1}
	B. Barrios, E. Colorado, A. de Pablo,  and U. S\'{a}nchez,
	On some critical problem for the fractional Laplacian operator,
	{\it J. Differ. Equ.},
	{\bf 252} (2012), 6133-6162.
	
	\bibitem{BLW}
	T. Bartsch,  Z. Liu,  T. Weth,
	Sign changing solutions of superlinear Schr\"{o}dinger equations.
	{\it Commun. Partial Differ.Equ.,}
	{\bf 29} (2004), 25--42.
	
	
	
	\bibitem{BW1}
	T. Bartsch, Z. Q. Wang,
	Existence and multiplicity results for some superlinear elliptic problems on $\R^N$,
	{\it Commun. Partial Differ. Equ.},
	{\bf 20} (1995), 1725-1741.
	
	\bibitem{BWe}
	T. Bartsch, T. Weth,
	Three nodal solutions of singularly perturbed
	elliptic equations on domains without topology,
	{\it Ann. Inst. H. Poincar\'{e}
		Anal. Non Lin\'{e}aire},
	{\bf 22} (2005), 259-281.
	
	\bibitem{BW2}	
	T. Bartsch, M. Willem,
	Infinitely many radial solutions of a semilinear elliptic problem on $\R^{N}$,
	{\it Arch. Ration. Mech. Anal.,}
	{\bf 124} (1993) 261--276.
	
	\bibitem{CS1}
	X. Cabr$\acute{e}$, Y. Sire,
	Nonlinear equations for fractional Laplacians II: Existence, uniqueness, and qualitative properties of solutions,
	{\it Tran. Amer. Math. Soc.},
	{\bf 367} (2011), 911-941.
	
	
	\bibitem{CT}
	X. Cabr$\acute{e}$, J. Tan,
	Positive solutions of nonlinear problems involving the square root of the Laplacian,
	{\it 	Adv. Math.},
	{\bf 224} (2010), 2052-2093.

	

	
	
	
	\bibitem{CS}
	L. Caffarelli, L. Silvestre,
	An extension problem related to the fractional Laplacian,
	{\it 	Commun. Partial Differ. Equ.},
	{\bf 32} (2007), 1245-1260.
	
%	\bibitem{CV}
%	L. Caffarelli, E. Valdinoci,
%	Uniform estimates and limiting arguments for nonlocal minimal surfaces,
%	{\it 	Calc. Var. Partial Differ. Equ.},
%	{\bf 41} (2011), 203-240.
	
%	\bibitem{CDS}
%	M. M. Cavalcanti, V. N. Domingos Cavalcanti, J. A. Soriano,
%	Global existence and uniform decay rates for the
%	Kirchhoff-Carrier equation with nonlinear dissipation,
%	{\it Adv. Differential Equations}
%	{\bf 6} (2001), 701-730.
	
%    \bibitem{CM}
%    S. Y. A. Chang, M. del Mar Gonz\'{a}lez,
%    Fractional Laplacian in conformal geometry,
%    {\it Adv. Math.,}
%    {\bf 226} (2010), 1410-1432.

    \bibitem{CW}
    X. Chang, Z. Q. Wang,
    Nodal and multiple solutions of nonlinear problems involving the fractional Laplacian,
    {\it J. Differ. Equ.,}
    {\bf 258} (2014), 2965-2992.


    \bibitem{CR}
    M. Chipot, J. F. Rodrigues, On a class of nonlocal nonlinear elliptic problems,
    {\it RAIRO ModšŠl. Math. Anal. NumšŠr.,}
    {\bf 26} (1992), 447-467.


	\bibitem{DS}
	P. D'Ancona, S. Spagnolo,
	Global solvability for the degenerate Kirchhoff equation with real analytic data,
	{\it Invent. Math.,}
	{\bf 108} (1992), 247-262.
	
	\bibitem{DPS}
	Y. B. Deng, S. J. Peng, W. Shuai,
	Existence and asymptotic behavior of nodal solutions for the Kirchhoff-type problems
	in $\R^3$,
	{\it J. Funct. Anal.,}
	{\bf 269} (2015), 3500-3527.
	
	\bibitem {DPV}
	E. Di Nezza, G. Palatucci, E. Valdinoci,
	Hitchhiker's guide to the fractional Sobolev spaces,
	{\it Bull. Sci. Math.,}
	{\bf 136} (2012), 521--573.
	
	

	\bibitem {FW}
	M. Fall, T. Weth,
	Nonexistence results for a class of fractional elliptic boundary value problems,
	{\it J. Funct. Anal.,}
	{\bf 263} (2012), 2205-2227.
	
	
	
%	\bibitem {FV}
%	A. Fiscella, E. Valdinoci,
%	A critical Kirchhoff type problem involving a nonlocal operator,
%	{\it Nonlinear Anal.,}
%	{\bf 44} (2014), 156-170.
	
	\bibitem {HZ}
    X. M. He, W. M. Zou,
    Infinitely many positive solutions for Kirchhoff-type problems,
    {\it Nonlinear Anal.,}
    {\bf 70} (2009), 1407-1414.

	
	
	
	\bibitem {K}
	G. Kirchhoff, Mechanik, Teubner, Leipzig, 1883.
	
    \bibitem {La1}
    N. Laskin,
    Fractional quantum mechanics and LšŠvy path integrals,
	{\it Phys. Lett. A,}
	{\bf 268} (2000), 298-305.
	
	\bibitem {La2}
	N. Laskin,
	Fractional Schr\"{o}dinger equation,
	{\it Phys. Rev.,}
	{\bf 66} (2002), 56-108.
	
	
	\bibitem {L}
   J. L. Lions, On some questions in boundary value problems of mathematical physics, in: Contemporary Developments
in Continuum Mechanics and Partial Differential Equations, in: North-Holland Math. Stud., vol.30, North-Holland,
Amsterdam, New York, (1978), 284--346.



	
	\bibitem {LY}
	G. B. Li, H. Y. Ye,
	Existence of positive ground state solutions for the nonlinear Kirchhoff type equations in $\R^3$,
	{\it J. Differ. Equ.,}
	{\bf 257} (2014), 566-600.
	
	
	\bibitem {MZ}
	A. M. Mao, Z. T. Zhang,
	Sign-changing and multiple solutions of Kirchhoff type problems without the P.S. condition,
	{\it Nonlinear Anal.,}
	{\bf 70} (2009), 1275-1287.

    \bibitem {ML}
    A. M. Mao, S. X. Luan,
    Sign-changing solutions of a class of nonlocal quasilinear elliptic boundary value problems,
    {\it J. Math. Anal. Appl.,}
    {\bf 383} (2011), 239-243.

	\bibitem {M1}
	 C. Miranda,
	 Un'osservazione su un teorema di Brouwer,
	 {\it Boll. Un. Mat. Ital.,} {\bf 3} (1940), 5--7.
	

	\bibitem {MRS}
	G. Molica Bisci, V. R\v{a}dulescu, R. Servadei,
	Variational Methods for Nonlocal Fractional Problems, with a
	Foreword by Jean Mawhin, Encyclopedia of Mathematics and its Applications, Cambridge University Press,
	{\bf 162}
	Cambridge, 2016.
	
	
	\bibitem{NW}
	E. S. Noussair, J. Wei,
	On the effect of the domain geometry on the existence and profile of nodal solution of some singularly perturbed semilinear Dirichlet problem.
	{\it Indiana Univ. Math. J.,}
	{\bf 46}  (1997), 1321--1332.
	
%	\bibitem {N}
%	N. Nyamoradi,
%	Existence of three solutions for Kirchhoff nonlocal operators of elliptic type,
%	{\it Math. Commum.,}
%	{\bf 18} (2013), 489-502.
	
	
	\bibitem {PXZ}
     P. Pucci, M. Q. Xiang, B. L. Zhang,
     Multiple solutions for nonhomogeneous Schr{\"{o}}dinger-Kirchhoff type equations involving the fractional p-Laplacian in
     $\R^{N}$,
     {\it Calc. Var. Partial Differ. Equ.,}
     {\bf 54} (2015), 1-22.




     \bibitem {RS}
     X. Ros-Oton, J. Serra,
     The Dirichlet problem for the fractional Laplacian: Regularity up to the boundary,
     {\it J. Math. Pures Appl.,}
     {\bf 101} (2012), 171-187.



     \bibitem {S}
     W. Shuai,
     Sign-changing solutions for a class of Kirchhoff-type problem in bounded domains,
     {\it J. Differ. Equ.,}
     {\bf 259} (2015), 1256-1274.

	\bibitem {SW}
	 W. A. Strauss,
	 Existence of solitary waves in higher dimensions.
	 {\it Comm. Math. Phys.,}
	 {\bf 55} (1977), 149-162.
	
	
	\bibitem{SV}
	R. Servadei, E. Valdinoci,
	The Brezis-Nirenberg result for the fractional Laplacian,
	{\it Trans. Am. Math. Soc.,}
	{\bf 367} (2015), 67-102.
	
	
	\bibitem{T}
	J. Tan,
	The Brezis-Nirenberg type problem involving the square root of the Laplacian,
	{\it Calc. Var. Partial Differ. Equ.,}
	{\bf 42} (2011), 21-41.
	
	\bibitem{W}
	T. Weth,
	Energy bounds for entire nodal solutions of autonomous superlinear equations,
	{\it  Calc. Var. Partial Differ. Equ.,}
	{\bf 27} (2006), 421--437.
	
 	\bibitem{WM}
 	 M. Willem, Minimax Theorems,  Birkh{\"{a}}user Basel, 1996.
 	
 	
% 	\bibitem{XB}
% 	M. Q. Xiang, G. M. Blscl, G. H. Tian, B. L. Zhang,
% 	Infinitely many solutions for the stationary Kirchhoff problems involving the fractional p-Laplacian,
% 	{\it Nonlinearity,}
% 	{\bf 29} (2016), 357-374.
	
	\bibitem{ZP}
	Z. T. Zhang, K. Perera,
	Sign changing solutions of Kirchhoff type problems via invariant sets of descent flow,
	{\it J. Math. Anal. Appl.,}
	{\bf 317} (2006), 456-463.


	
	
	
		

	
	
	
	
	

	

	
	
	
	
	
	
	
	
	
	
	
	
	
	
	
	
	
	
	
	
	
\end{thebibliography}
\end{document}